\documentclass[oneside,english]{article}
\usepackage[T1]{fontenc}
\usepackage{amsthm}
\usepackage{amssymb}
\usepackage[all]{xy}
\usepackage{amsmath}
\usepackage{comment}
\usepackage{color}
\usepackage{mathtools}
\usepackage{tikz-cd}

\makeatletter
\newcommand{\nc}{\newcommand}
\newtheorem{thm}{Theorem}
\theoremstyle{plain}
\nc{\bthm}{\begin{thm}} \nc{\ethm}{\end{thm}}
\newtheorem{prop}[thm]{Proposition}
\nc{\bprp}{\begin{prop}} \nc{\eprp}{\end{prop}}
\newtheorem{fact}[thm]{Fact}
\nc{\bfct}{\begin{fact}} \nc{\efct}{\end{fact}}
\newtheorem{prob}[thm]{Problem}
\nc{\bprb}{\begin{prob}} \nc{\eprb}{\end{prob}}
\newtheorem{lem}[thm]{Lemma}
\nc{\blem}{\begin{lem}} \nc{\elem}{\end{lem}}
\newtheorem{claim}[thm]{Claim}
\nc{\bclm}{\begin{claim}} \nc{\eclm}{\end{claim}}
\newtheorem{cor}[thm]{Corollary}
\nc{\bcor}{\begin{cor}} \nc{\ecor}{\end{cor}}
\newtheorem{conj}[thm]{Conjecture}
\nc{\bcnj}{\begin{conj}} \nc{\ecnj}{\end{conj}}
\theoremstyle{definition}
\newtheorem{defn}[thm]{Definition}
\nc{\bdfn}{\begin{defn}} \nc{\edfn}{\end{defn}}
\newtheorem{observation}[thm]{Observation}
\nc{\bobs}{\begin{observation}} \nc{\eobs}{\end{observation}}
\theoremstyle{remark}
\newtheorem{rem}[thm]{Remark}
\nc{\brem}{\begin{rem}} \nc{\erem}{\end{rem}}
\newtheorem{cnv}[thm]{Convention}
\nc{\bcnv}{\begin{cnv}} \nc{\ecnv}{\end{cnv}}
\newtheorem{exam}[thm]{Example}
\nc{\bexm}{\begin{exam}} \nc{\eexm}{\end{exam}}

\nc{\bpf}{\begin{proof}} \nc{\epf}{\end{proof}}
\nc{\be}{\begin{enumerate}}
	\nc{\ee}{\end{enumerate}}
\nc{\bi}{\begin{itemize}}
	\nc{\itm}{\item}
	\nc{\ei}{\end{itemize}}
\nc{\invlim}{\lim_{\leftarrow}}
\nc{\dirlim}{\lim_{\rightarrow}}
\nc{\mm}{\mathbf{m}}
\nc{\FF}{\mathcal{F}}

\title{Cohomological properties of absolute Galois groups and their profinite completion.}
\author{Tamar Bar-On}
\date{\today}
\begin{document}
	
	\maketitle
	\begin{abstract}
We prove that several properties of absolute Galois groups are preserved under a profinite completion. 
	\end{abstract}
\section*{Introduction}
	Profinite groups were first introduced in 1928 by Wolfgang Krull, who proved that regarded as a topological group with the Krull topology, every Galois group of a field extension is compact, Hausdorff and totally disconnected,  (\cite{krull1928galoissche}). In the 1950s Horst Leptin proved that every profinite group can be realized as the Galois group of some field extension (\cite{leptin1955darstellungssatz}). One of the main questions in Galois theory these days is computing the absolute Galois groups of fields, and of main importance, the field of rational numbers, $\mathbb{Q}$ (see, for example, \cite{girondo2007note}, \cite{bauer2015faithful} and \cite{bary2015sylow}). The opposite direction of this question is asking to identify absolute Galois groups among all profinite groups. Several results are already known: for example, the only finite group which can be realized as an absolute Galois group is $C_2$. In addition, for every infinite cardinal $\mm$, the free profinite group of rank $\mm$ is isomorphic to the absolute Galois group of $F(t)$, were $F$ is an algebraically closed field of cardinality $\mm$ (\cite{douady1964determination}). However, the known cases are just a drop in the ocean, and determining whether a given profinite group can be realized as an absolute Galois group is still considered as an open question. A lot of effort is devoted to find restrictions on the possible structure and properties of an absolute Galois group. Recall that the absolute Galois group of a field $F$ comes equipped with a natural $\bar{F}$-module structure, and its special submodules $\mu_p(F)$, the set of all $p$ -routes of unity, which are tightly connected to the cohomological groups of $G_F=\operatorname{Gal}(\bar{F}/F)$. Hence, most of these restrictions regard the structure of modules over $G$ and the properties of cohomology groups of $G$ over specific modules. In order to make things a bit simpler, we often assume that the base field $F$ contains all $p$ -roots of unity. Thus, the natural action of $G_F$ on $\mu_p(F)\cong \mathbb{F}_p$ is trivial. One should notice that if $G_F=\operatorname{Gal}(\bar{F}/F)$ is the absolute Galois group of $F$, and $H\leq_cG$ is a closed subgroup of $G$, then $H=\operatorname{Gal}(\bar{F}/K)$ where $K=\bar{F}^H$ denotes the field of fixed points of $H$. Thus, $H$ is isomorphic to an absolute Galois group as well. Moreover, if $F$ contains all $p$-roots of unity then so does $K$. As a result, if $\mathcal{F}$ is a property characterizes absolute Galois groups, then we will always be interested in the property $\mathcal{F}'$ which says that "every closed subgroup of $G$ satisfies $\mathcal{F}$".  
	
We want to investigate the connection between $G$ and $\hat{G}$ having the required properties of absolute Galois groups. 
\begin{defn}[\cite{ribes2000profinite}, Section 3.2]
	Let $G$ be an abstract group. The profinite completion of $G$, denoted by $\hat{G}$, is a profinite group equipped with a homomorphism $i:G\to \hat{G}$ which satisfies the following universal property: for every homomorphism $f:G\to H$ where $H$ is a profinite groups, there exits a unique continuous homomorphism $\hat{f}:\hat{G}\to H$ which makes the following diagram commutative:
		\[
	\xymatrix@R=14pt{ \hat{G}\ \ar[rd]^{\hat{f}}& \\
	G\ \ar[u]_{i} \ \ar[r]_f& H\\
	}
	\]
	
\end{defn}
Let $G$ be a profinite group. Regarded as an abstract group, $G$ also possess a profinite completion. Despite looking a bit confusing in the first place, a profinite group does not necessarily equal to its profinite completion. For a counterexample, one may look, for example, at \cite[Example 4.2.12]{ribes2000profinite}. A profinite group which equals to its profinite completion is called \textit{strongly complete}. This is equivalent to the condition that every subgroup of finite index is open.

One of the main theorems in the whole area of profinite groups is the celebrated solution of Segal of Nikolov to Serre's conjecture, which states that every finitely generated profinite group is strongly complete (\cite{nikolov2007finitely}).

In fact, there are very "few" profinite groups which are strongly complete, and these are only the "small" profinite groups. This is due to the following proposition (\cite{smith2003subgroups}):
\begin{prop}
	Let $G$ be a profinite group, then the following are equivalent:
	\begin{itemize}
		\item $G$ is strongly complete.
		\item For every natural number $n$, $G$ has only finite number of subgroups of index $n$.
	\end{itemize}
\end{prop}
Thus, many absolute Galois groups are nonstrongly complete, such as $G_{\mathbb{Q}}$ and all free profinite groups of infinite rank.

In \cite{bar-on_2021} the author proved the following:
\begin{thm}
	Let $G$ be a nonstrongly complete profinite projective group. Then, $\hat{G}$ is profinitely projective.
\end{thm}
By \cite{fried2006field} Theorem 11.6.2 and Corollary 23.1.22, a profinite group is projective if and only if it can be realized as the absolute Galois group of a pseudo algebraically closed field. So, this gives a family of examples of absolute Galois groups which their profinite completion can be realized as an absolute Galois group as well. Notice that this also implies a cohomological connection between a profinite group and its profinite completion, due to the fact that being projective is equivalent to cohomological dimension $\leq 1$.

The goal of this paper is to prove that for some of known cohomological properties of absolute Galois groups, the family of profinite groups which satisfy this property is closed under taking a profinite completion.  
\section{Main Results}
In this section we review some of the known properties of absolute Galois groups, and show that they are preserved under profinite completion. But first we need a few lemmas that will be used repetitively in the proofs. 
\begin{lem}\label{locally_closed}
	Let $H$ be a finite index subgroup of a profinite group $G$. Then $H$ is locally closed. I.e, for every $x_1,..,x_n\in H$, $K=\overline{\langle x_1,...,x_n \rangle}\subseteq H$.
\end{lem}
\begin{proof}
	Let $x_1,...,x_n\in H$. Look at $K=\overline{\langle x_1,...,x_n \rangle}$, the closed subgroup generated by $x_1,..,x_n$. $K\cap H$ is a finite index subgroup of $K$. Since $K$ is a finitely generated profinite group, by Segal\&Nikolov theorem $K$ is strongly complete, so $K\cap H$ is open in $K$. But $K\cap H$ contains the generators $x_1,...,x_n$ of $K$, so $K\subseteq K\cap H\Rightarrow K\subseteq H$.
\end{proof}
The following generalization is not needed for this paper, but it is also interesting to mention out:
\begin{lem}
	Let $H$ be a finite index normal subgroup of a profinite group $G$. Then $H$ is locally closed-normal.
\end{lem}
\begin{proof}
	Let $x_1,...,x_n\in H$. We would like to show that the closed normal closure of $\{x_1,...,x_n\}$, i.e. the minimal closed normal subgroup that contains $x_1,..,x_n$, is contained in $H$. This group equals to $\bigcap U$ for all $x_1,...,x_n\in U\unlhd_oG.$ Suppose not, then there is an element $y\in (\bigcap U)\setminus H$. Let us look at $K=\overline{\langle x_1,..,x_n, y\rangle }$. $H\cap K$ is a normal finite index subgroup of $K$ that contains $x_1,..,x_n$. Since $K$ is f.g, it is strongly complete, so $H\cap K$ is open in $K$. By \cite[Lemma 1.2.5(b)]{fried2006field}, there exist an open normal subgroup $O\unlhd_oG$ such that $K\cap H=K\cap O$. In particular, $O$ is an open normal subgroup of $G$ that contains $x_1,...,x_n$. Since $y$ belongs to the intersection of all such subgroups, $y\in O$. But $y\in K$ implies that $y\in K\cap O=K\cap H$. Thus, $y\in H$. A contradiction. 
\end{proof}
Recall that, for every abstract residually finite group $G$, there is a one-to-one correspondence between the open subgroups of $\hat{G}$ and the finite index subgroups of $G$, obtained by $U\to \bar{U}$ where $U$ denotes a finite-index subgroup of $G$, and $\bar{U}$ denotes its topological closure in $\hat{G}$ (\cite[Proposition 3.2.2]{ribes2000profinite}). Moreover, since every finite-index subgroup of $U$ is also finite index in $G$, $\bar{U}\cong \hat{U}$ (\cite[Proposition 3.2.5+3.2.6]{ribes2000profinite}). We also need the following Lemma:
\begin{lem}\label{from_closed_to_open}
	Let $G$ be a profinite group, $H$ a closed subgroup of $G$, and $\varphi:H\to A$ a continuous homomorphism to a finite group. Then $\varphi$ can be lifted to some open subgroup $U$ of $G$. 
\end{lem}  
\begin{proof}
	This is part of the proof of Lemma 7.6.3 in \cite{ribes2000profinite}.
\end{proof}
During this paper, we are going to use intensively the celebrated solution to Serre's conjecture, which is also known as Nikolov\&Segal Theorem.
\begin{thm}\cite{nikolov2007finitely}
	Every finitely generated profinite group is strongly complete.
\end{thm} 
\begin{rem}
	Notice that this immediately implies that every abstract homomorphism from a finitely generated profinite group to an arbitrary profinite group is continuous. 
\end{rem}
In addition, we are going to use the following properties of profinite completion several times:
\begin{lem}\cite[Chapter 3.1]{ribes2000profinite}
	Let $G$ be an abstract group and $\hat{G}$ its profinite completion.
	\begin{enumerate}
		\item $i(G)$ is dense in $\hat{G}$.
		\item $i:G\to \hat{G}$ is one-to-one iff $G$ is residually finite.
	\end{enumerate}
\end{lem}
\begin{cor}
	By \cite[Theorem 2.1.3]{ribes2000profinite} the intersection of all open normal subgroups of a profinite group is trivial, which means that a profinite group is always residually finite. Thus, given a profinite group $G$, we can consider it as a dense subgroup of its profinite completion $\hat{G}$.
\end{cor}
Eventually, the following proposition is a key tool in our toolbox:
\begin{prop}\label{key tool}
	Let $\mathcal{C}$ be a property of functions from a group to a given finite group $A$ which preserved under reduction and union. Let $U$ be a finite index subgroup of some profinite group $G$. Assume that for all finitely generated closed subgroup $H$ of $U$ there is a continuous homomorphism $H\to A$ satisfying $\mathcal{C}$. Then there is an abstract homomorphism $U\to A$ satisfying $\mathcal{C}$.
\end{prop}
\begin{proof}
	For all finitely generated closed subgroup $H\leq_cU$ denote by $\mathcal{C}_H$ the set of all continuous homomorphisms $H\to A$ satisfying $\mathcal{C}$. Since $H$ is f.g and $A$ finite, $\operatorname{Hom}(H,A)$ is finite. Hence so is $\mathcal{C}_H$. For all $H\leq K$ finitely generated closed subgroups of $U$ we can define a map $\varphi_{KH}:\mathcal{C}_K\to \mathcal{C}_H$ by $f\to f|_H$, since the property preserved under reduction of the domain. Notice that the finitely generated closed subgroups of $U$ ordered by inclusion form a directed set. Thus, $\{\mathcal{C}_H,\varphi_{KH}\}$ is an inverse system of finite sets. Thus, by $\cite[Proposition 1.1.4]{ribes2000profinite}$ its inverse limit is nonempty. Notice that an element in the inverse limit is a homomorphism $\bigcup H\to A$. By Lemma \ref{locally_closed} $\bigcup H=U$. Eventually, by assumption that $\mathcal{C}$ is closed under union, we get that there is a homomorphism $U\to A$ satisfying $\mathcal{C}$ as required. 
\end{proof}
	\subsection{Massey products}
	The first property we deal with is the vanishing of the $n$-fold Massey products. Massey products were defined in a much wider context, and were first introduced by Massey in \cite{massey1958some}, where he proved that Borromean rings are not equivalent to three unlinked circles by showing that the singular cochain complex of the complement of the Borromean rings in $\mathbb{R}^3$ admits a non-trivial triple Massey product. In this paper we will only present Massey products in the context of Galois cohomology of profinite groups.
	
	Let $G$ be a profinite group, and $a_1,...,a_n\in H^1(G,\mathbb{F}_p)$. A defining system for $a_1,...,a_n$ is a set $a_{ij}\in C^1(G,\mathbb{F}_p)$ for $1\leq i\leq j\leq n$, $(i,j)\ne (1,n)$ which satisfies:
	\begin{enumerate}
		\item For all $i$, $[a_{ii}]=a_i$
		\item $\partial a_{ij}=\sum _{r=i}^{j-1}a_{ir}a_{r+1,j}$ for all $1\leq i\leq j\leq n$, $(i,j)\ne (1,n)$. 
	\end{enumerate}
	The Massey product $\langle a_1,...,a_n\rangle$ is the subset of $H^2(G,\mathbb{F}_p)$ consist of all cohomology classes $[\sum _{r=1}^{n-1}a_{1r}a_{r+1,n}]$ obtained from defining systems for $a_1,...,a_n$. We say that the Massey product is vanishing if it contains $0$.
	
	In case $G$ acts trivially on $\mathbb{F}_p$ we have the following equivalence:
	
	Denote by $U_{n+1}(\mathbb{F}_p)$ the group of $n+1\times n+1$ unitriangular matrices over the ring $\mathbb{F}_p$, and by $\overline{U_{n+1}(\mathbb{F}_p)}$ the quotient of $U_{n+1}(\mathbb{F}_p)$ by the $(1,n+1)$ entry, which is equivalent to the group of unitriangular matrices with the $(1,n+1)$ entry omitted. Let $G$ be a profinite group acting trivially on $\mathbb{F}_p.$ Then, $G$ has the vanishing $n$- Massey product property iff for every (continuous) homomorphism $\varphi:G\to \overline{U_{n+1}(\mathbb{F})_p}$ there exists a (continuous) homomorphism $\psi:G\to U_{n+1}(\mathbb{F}_p)$ such that for every element $g\in G$ and $1\leq i\leq n$ $[\psi(g)]_{i,i+1}=[\varphi(g)]_{i,i+1}$.
	\[
	\xymatrix@R=14pt{ & & &G \ar@{->}[dd]^(0.3){\varphi}  
		\ar@{.>}[]!<0ex,0ex>;[dl]!<2ex,-3ex>& \\
		&&&&\\
		1 \ar[r] & \mathbb{F}_p \ar[r] & {\begin{bsmallmatrix}
				1&\rho_{1,2}&\rho_{1,3}&...&\rho_{1,n}\\
				&1&\rho_{2,3}&...&\rho_{2,n}\\
				&&\ddots&\ddots&\vdots \\
				&&&1&\rho_{n-1,n}\\
				&&&&1\\
		\end{bsmallmatrix}} \ar[r]^{\alpha}&{\begin{bsmallmatrix}
				1&\rho_{1,2}&\rho_{1,3}&...&\\
				&1&\rho_{2,3}&...&\rho_{2,n}\\
				&&\ddots&\ddots&\vdots \\
				&&&1&\rho_{n-1,n}\\
				&&&&1\\
		\end{bsmallmatrix}}\ar[r]&1\\
	}
	\]

	For more convenience, we call such a homomorphism $\psi:G\to U_{n+1}(\mathbb{F}_p)$ a "twisted solution" to the diagram. It has been conjectured by Min\'{a}\v{c} and T\^{a}n that for every $n$ and every field $F$ contains a primitive root of unity, $G_F$ satisfies the $n$- vanishing Massey product property ( see \cite{minavc2016tripleingalois}. In the paper \cite{minavc2016triple} was this conjecture in fact extended to all fields .). This conjecture is known to be true in the following cases:
	\begin{itemize}
		\item $n=3$, $p=2$ and $F$ is arbitrary (\cite{minavc2016triple}).
		\item $n=3$, $p$ is odd, and $F$ is arbitrary (\cite{matzri2014triple}. See also \cite{efrat2017triple}).
		\item $n=4$, $p=2$ and $F$ is arbitrary (\cite{merkurjev2023massey}).
		\item $F$ is a number field and $n$ and $p$ are arbitrary (\cite{harpaz2023massey}). 
	\end{itemize} 
	\begin{thm}\label{main theorm 1}
		Let $G$ be a profinite group acting trivially on $\mathbb{F}_p$ such that for every closed subgroup $H\leq_cG$, $H$ satisfies the $n$- vanishing Massey product property. Then so does every closed subgroup of $\hat{G}$.
	\end{thm}
	We prove the theorem in several steps.
	\begin{prop}\label{abstract_massey_products}
		Let $G$ be an abstract group and consider $\mathbb{F}_p$ as a trivial $\hat{G}$- module. Then $\hat{G}$ satisfying the $n$- vanishing Massey product property iff for every \textbf{abstract} homomorphism $\varphi:G\to \overline{U_{n+1}(\mathbb{F})_p}$ there exists an \textbf{abstract} homomorphism $\psi:G\to U_{n+1}(\mathbb{F}_p)$ such that for every element $g\in G$ and $1\leq i\leq n$ $[\psi(g)]_{i,i+1}=[\varphi(g)]_{i,i+1}$. 
	\end{prop}
	\begin{proof}
		$\Rightarrow$ Let $\varphi: G\to \overline{U_{n+1}(\mathbb{F})_p}$ be an abstract homomorphism. By definition of the profinite completion it can be lifted to a continuous homomorphism $\bar{\varphi}:\hat{G}\to \overline{U_{n+1}(\mathbb{F})_p}$. By assumption, there exists a twisted (continuous) solution $\psi:\hat{G}\to U_{n+1}(\mathbb{F}_p)$. Then $\psi|_G:G\to U_{n+1}(\mathbb{F}_p)$ is an abstract twisted solution for $G$.
		
		$\Leftarrow$ 
		Let $\varphi:\hat{G}\to \overline{U_{n+1}(\mathbb{F})_p} $ be a continuous homomorphism. Its reduction to $G$ is an abstract homomorphism $\varphi|_G:G\to \overline{U_{n+1}(\mathbb{F})_p}$. By assumption $G$ admits an abstract twisted solution $\psi:G\to U_{n+1}(\mathbb{F}_p)$. We claim that its continuous lifting $\bar{\psi}:\hat{G}\to U_{n+1}(\mathbb{F}_p)$ is a twisted solution. Indeed, for every $1\leq i<n$ the maps $[\bar{\psi}]_{i,i+1}$ and $[\varphi]_{i,i+1}$ which goes to an Hausdorff space, identify on the dense subset $G$, thus they are equal.
	\end{proof}
	\begin{lem}
		In order to prove the property holds for every closed subgroup of $\hat{G}$, it is enough to prove so for every open subgroup.
	\end{lem}
	\begin{proof}
		Let $H\leq_cG$ and $\varphi:H\to  \overline{U_{n+1}(\mathbb{F})_p}$ a continuous homomorphism. By Lemma \ref{from_closed_to_open} there exists an open normal subgroup $O\unlhd_cG$ such that $\varphi$ can be lifted to a continuous homomorphis $\bar{\varphi}:O\to \bar{U_{n+1}(\mathbb{F})_p}$. Let $\psi:O\to U_{n+1}(\mathbb{F})_p $ be a "twisted" solution to the corresponding embedding problem, then $\psi|_H$ is a twisted solution for $H$.
	\end{proof}
	
	\begin{proof}[Proof of Theorem \ref{main theorm 1}]
	Let $U$ be a finite index subgroup of $G$, and $\varphi:U\to \overline{U_{n+1}(\mathbb{F})_p}$ an abstract homomorphism. By Proposition \ref{abstract_massey_products} it is enough to prove that $\varphi$ admits an abstract twisted solution. For every finite subset $D=\{x_1,...,x_n\}$ of $U$, look at the reduction of $\varphi$ to $H_D=\overline{\langle D\rangle }$. Since $H_D$ is finitely generated, by Nikolov\&Segal Theorem $\varphi|_{H_D}$ is continuous. So, since $H_D$ is a closed subgroup of $G$, by assumption, it admits a twisted solution $\psi_D$. It is easy to see that if $D\subseteq E$ then any twisted solution from $H_E$ induces a twisted solution from $H_D$ by reduction of the domain, and that the union of twisted solutions is a twisted solution. Thus, by Proposition \ref{key tool} there exists an abstract twisted solution from $U$. In conclusion, $\hat{U}\cong \bar{U}$ satisfies the $n$- fold vanishing Massey product property. Since these are all the open subgroups of $\hat{G}$, by Lemma \ref{from_closed_to_open} we are done. 
	\end{proof}
	\subsection{Cyclotomic orientation}
	Let $p$ be a prime and $G$ be a profinite group acting trivially on $\mathbb{F}_p$. A $p$- orientation of $G$ is a homomorphism $\theta:G\to \mathbb{Z}_p^{\times}$. A profinite group together with a $p$- orientation is called a $p$- oriented profinite group. Later on we will omit the $p$. We will always assume that this orientation induces the trivial action on $\mathbb{F}_p$. Notice that the orientation makes $\mathbb{Z}_p$ a $\mathbb{Z}_p[G]-$ module, by defining $g.x=\theta(g)x$. We denote this module by $\mathbb{Z}_p(1)$.
	
	Let $F$ be a field containing a primitive root of unity. Denote by $\mu_{p^\infty}(F)$ the subgroup of $\bar{F}$ contains of all power $p$ roots of unity. Then, $\operatorname{Aut}(\mu_{p^\infty}(F))\cong \mathbb{Z}_p^{\times}$. Thus, the natural action of $G_F$ on $\bar{F}$ makes $G_F$ as an oriented profinite group. This induced orientation is called the "arithmetical orientation". The arithmetical orientation of $G_F$ admits a special property: 
	
	By the exact sequence 
		\[
	\xymatrix@R=14pt{ 1\ \ar[r]& \mu_{p^n}(F)\ \ar[r]& {\bar{F}}^{\times}\ \ar[r]^{p^n}& {\bar{F}}^{\times}\ \ar[r]& 1 \\
		}
	\]
	we get that ${\bar{F}}^{\times}/(\bar{F})^{\times {p^n}}\cong H^1(G_F,\mu_{p^n}(F))\cong H^1(G_F,\mathbb{Z}_p(1)/p^n)$. Thus the natural epimorphism ${\bar{F}}^{\times}/(\bar{F})^{\times {p^n}}\to {\bar{F}}^{\times}/(\bar{F})^{\times {p}}$ implies that the natural map $H^1(G_F,\mathbb{Z}_p(1)/p^n)\to H^1(G_F,\mathbb{Z}_p(1)/p)$ is in fact an epimorphism. That was the motivation to the following definition:
	\begin{defn}[\cite{quadrelli2021chasing}]
		Let $G$ be an oriented profinite group such that the induced action on $\mathbb{F}_p$ is trivial. Then, $G$ is called a \textbf{cyclotomic} oriented profinite group if the induced module structure on $\mathbb{Z}_p$ satisfies that the natural map $$H^1(H,\mathbb{Z}_p(1)/p^n\mathbb{Z}_p(1))\to H^1(H,\mathbb{Z}_p(1)/p\mathbb{Z}_p(1))$$ is surjective, for all closed subgroup $H\leq_cG$ and for every natural number $n$.
	\end{defn}
\begin{cor}
	Let $F$ be a field containing a primitive root of unity and $G_F$ be its absolute Galois group, then equipped with the arithmetical orientation, $G_F$ is a cyclotomic oriented profinite group.  
\end{cor}
\begin{lem}\label{H2_trivial}
	An oriented profinite group is cyclotomic iff for every $n$ and every closed subgroup $H\leq_cG$, the kernel of the natural map $$H^2(H,\mathbb{Z}_p(1)/p^{n-1}\mathbb{Z}_p(1))\to H^2(H,\mathbb{Z}_p(1)/p^n\mathbb{Z}_p(1))$$ which induced by the short exact sequence 
	\[
	\xymatrix@R=14pt{ 1\ \ar[r]& \mathbb{Z}_p/p^{n-1}\ \ar[r]&\mathbb{Z}_p/p^n \ \ar[r]& \mathbb{Z}_p/p\ \ar[r]& 1 \\
	}
	\]
	is trivial.
\end{lem}
\begin{proof}
	Look at the exact sequence 
		\[
	\xymatrix@R=14pt{ 1\ \ar[r]& \mathbb{Z}_p/p^{n-1}\ \ar[r]&\mathbb{Z}_p/p^n \ \ar[r]& \mathbb{Z}_p/p\ \ar[r]& 1 \\
	}
	\]
	By the corresponding long exact sequence, the map $H^1(H,\mathbb{Z}_p(1)/p^n\mathbb{Z}_p(1))\to H^1(H,\mathbb{Z}_p(1)/p\mathbb{Z}_p(1))$ is surjective if and only if the connecting homomorphism $\delta:H^1(H,\mathbb{Z}_p(1)/p\mathbb{Z}_p(1))\to H^2(H,\mathbb{Z}_p(1)/p^n\mathbb{Z}_p(1))$ is trivial, if and only if the induced map $H^2(H,\mathbb{Z}_p(1)/p^n\mathbb{Z}_p(1))\rightarrow H^2(H,\mathbb{Z}_p(1)/p^n\mathbb{Z}_p(1))$ is injective. 
\end{proof}
\begin{rem}
	Since $G$ acts trivially on $\mathbb{F}_p$, $H^1(H,\mathbb{Z}_p/p\mathbb{Z}_p)=\operatorname{Hom}(H,\mathbb{Z}_p/p\mathbb{Z}_p)$.
\end{rem}
	\begin{thm}\label{Main theorem 2}
		Let $G$ be a cyclotomic oriented profinite group. Then so is $\hat{G}$.
	\end{thm}
	We prove it by several steps.
	
	\begin{observation}
		The orientation of $G$ induces an orientation of $\hat{G}$, such that the induced module $\mathbb{Z}_p(1)/p\mathbb{Z}_p(1)$ is trivial.
	\end{observation}
	\begin{proof}
		By the universal property of $\hat{G}$, the homomorphism $\theta:G\to \mathbb{Z}_p^{\times}$ can be lifted uniquely to a continuous homomorphism $\bar{theta}:\hat{G}\to \mathbb{Z}_p^{\times}$. By assumption, $\operatorname{Im}(G)\leq 1+p\mathbb{Z}_p^{\times}$ which is closed and thus contains $Im(\hat{G})$ too.
	\end{proof}
	\begin{lem}\label{from_closed_to_open_2}
		It is enough to prove the surjectivity for all open subgroups of $\hat{G}$.
	\end{lem}
	\begin{proof}
		Let $n$ be a natural number, $H$ a closed subgroup of $G$ and $f\in H^1(H,\mathbb{Z}_p/p\mathbb{Z}_p)=\operatorname{Hom}(H,\mathbb{Z}_p/p\mathbb{Z}_p)$ . By Lemma \ref{from_closed_to_open}, any homomorphism from $H$ to $\mathbb{Z}_p/p$ can by lifted to a homomorphism from some open subgroup $U$ of $G$. By assumption there exist a crossed homomorphism $\alpha\in H^1(U,\mathbb{Z}_p(1)/p^n)$ with projects on this homomorphism. So the restriction of $\alpha$ to $H$ will do.
	\end{proof}
	\begin{proof}[Proof of Theorem \ref{Main theorem 2}]
		
		First we lift the orientation of $G$ to $\hat{G}$. 
		
		Let $n$ be a natural number. We will prove that the kernel of the induced map $H^2(\hat{U},\mathbb{Z}_p/p^{n-1})\to H^2(\hat{U},\mathbb{Z}_p/p^{n})$ is trivial, for every finite index subgroup $U\leq_f G$. This will prove that every open subgroup of $\hat{G}$ satisfies the required surjectivity, and thus by Lemma \ref{from_closed_to_open_2} we will get that $\hat{G}$ is cyclotomic.  
		
		Let $\alpha\in H^2(\hat{U},\mathbb{Z}_p/p^{n-1})$ be a cocycle which lays in the kernel. Look at its restriction to $U$. This is a cocycle in the noncontinuous second cohomology group of $U$ which lays in the kernel of the induced map $H^2_{abs}(U,\mathbb{Z}_p/p^{n-1})\to H^2_{abs}(U,\mathbb{Z}_p/p^n)$. So, we wish to prove that this kernel is trivial as well. By identical proof of Lemma \refeq{H2_trivial} this is equivalent to the surjectivity of $H^1_{abs}(U,\mathbb{Z}_p(1)/p^n\mathbb{Z}_p(1))\to H^1_{abs}(U,\mathbb{Z}_p(1)/p\mathbb{Z}_p(1))$. So, let $$f\in H^1_{abs}(U,\mathbb{Z}_p(1)/p\mathbb{Z}_p(1))=Hom_{abs}(U,\mathbb{Z}_p(1)/p\mathbb{Z}_p(1))$$ By Nikolov\&Segal Theorem, for all finitely generated closed subgroup $H$ of $U$, the restriction of $f$ is continuous, and thus by cyclotomicy of $G$ admits an origin which is a crossed homomorphism $\varphi\in H^1(H,\mathbb{Z}_p(1)/p^n)$. The number of continuous crossed homomorphisms from a f.g profinite group to a finite module is finite since it depends only on the values of the generators. Also, the restriction of such an origin is an origin as well, and so is the union, thus by identical proof of Proposition \ref{key tool}, $f$ admits an origin in $H^1(U,\mathbb{Z}_p(1)/p^n)$.
		
		Thus, $\alpha|_U$ is trivial. We want to show that $\alpha$ is trivial too. Notice that the origin of $\alpha|_U$ in $C^1(U,\mathbb{Z}/p^{n-1})$, which makes it trivial, can not trivially be lifted to $\hat{U}$ since it is not a homomorphism. Let 
		\[
		\xymatrix@R=14pt{ 1\ar[r]& \mathbb{Z}/p^{n-1}\ar[r]& K\ar[r] & \hat{U} \ar[r] &1
		}
		\]
		be the extension corresponding to $\alpha$. We shall show that $K\cong  \mathbb{Z}/p^n \rtimes \hat{U}$ with the given operation of $\hat{U}$ on $ \mathbb{Z}/p^{n-1}$. Let 
		\[
		\xymatrix@R=14pt{ 1\ar[r]& \mathbb{Z}/p^n\ar[r]& C\ar[r] & U \ar[r] &1
		}
		\]
		be the extension corresponding to $\alpha|_U$. By assumption $C\cong \mathbb{Z}/p^{n-1}\rtimes U$. Since the operation of $\hat{U}$ on $\mathbb{Z}/p^{n-1}$ is the extension of this of $U$, one easily checks that $\widehat{\mathbb{Z}/p^{n-1}\rtimes G}\cong \mathbb{Z}/p^n\rtimes \hat{U}$. Thus, the only thing left to show is that $\hat{C}\cong K$. Since $C$ is residually finite it induces on $\mathbb{Z}/p^{n-1}$ its complete profinite topology, so the exact sequence remains exact after taking profinite completion (See \cite[Proposition 3.2.5 and Lemma 3.2.6]{ribes2000profinite}).
		
		\[
		\xymatrix@R=14pt{ 1\ar[r]& \mathbb{Z}/p^{n-1}\ar[r]& \hat{C}\ar[r] & \hat{U} \ar[r] &1
		}
		\]    
		Let $\beta$ be the corresponding element in $H^2(\hat{U},\mathbb{Z}/p^{n-1})$. We want to show that $\alpha=\beta$. Since $U\times U$ is dense in $\hat{U}\times \hat{U}$ it is enough to show that $\alpha|_U=\beta|_U$.
		
		Recall that $\beta(g_1,g_2)=\sigma(g_1)+\sigma(g_2)-\sigma(g_1g_1)$ when $\sigma:\hat{U}\to \hat{C}$ is any continuous section.
		
		We will construct a continuous section $\sigma:\hat{U}\to \hat{C}$ whose restriction to $U$ will be a section from $C$ to $U$. For that, choose $O$ to be an open subgroup of $\hat{C}$ such that $O\cap \mathbb{Z}/p^{n-1}=\{e\}$. The same holds for $O\cap C\leq C$. Thus, $O\cap C\cong \operatorname{Im}(O\cap C)$ and $O=\overline{O\cap C}\cong \widehat{\operatorname{Im}(O\cap C)}=\operatorname{Im}(O)$. So we can chose a continuous isomorphism $\varphi:\operatorname{Im} (O)\to O$ which sends $\operatorname{Im}(O\cap C)\to U\cap C$. Now choose a transversal $\{g_1,..,g_n\}$ of $\operatorname{Im}(O\cap C)$ in $U$. By \cite[Proposition 3.2.2]{ribes2000profinite} this is also a transversal of $O$ in $\hat{U}$. For each $g_i$ chose an origin $c_i$ in $C$ and define $\sigma(g_iu)=c_i\varphi^{-1}(u)$. We get that $\sigma|_U$ is a section from $U$ to $C$. So for all $g_1,g_2\in U$   $\beta(g_1,g_2)=\sigma(g_1)+\sigma(g_2)-\sigma(g_1g_1)=\alpha|_U(g_1,g_2)$. Since $\beta$ and $\alpha$ are functions to a Hausdorff space which identify on the dense subset $U\times U$, they are equal. In conclusion, we get that $\alpha$ is trivial. 
	\end{proof}

\subsection{The cup-product exact sequence}
Let $G=G_F$ be the absolute Galois group of a field $F$ containing a primitive $p$'th root of unity. For every $\chi\in H^1(G,\mathbb{F}_p)=\operatorname{Hom}(G,\mathbb{F}_p)$ one has the following identities: $H^1(G,\mathbb{F}_p)\cong F^{\times}/(F^{\times})^p$,  $H^1(\ker(\chi),\mathbb{F}_p)\cong L^{\times}/(L^{\times})^p$, $H^2(G,\mathbb{F}_p)\cong \operatorname{Br}(F)_p$ and $H^2(\ker(\chi),\mathbb{F}_p)\cong \operatorname{Br}(L)$ where $L=\bar{F}^{\ker(\chi)}$ and $\operatorname{Br}(-)_p$ stands for the $p$ - torsion part of the Brauer group over a given field. Thus, the known isomorphism $F^{\times}/N_{L/F}(L^{\times})\cong Br(L/F)$ gives rises to the following exact sequence (for the proof look at \cite[pg. 73 Theorem 1]{draxl1983skew}): 

	\begin{tikzcd}
		H^1(\ker(\chi),\mathbb{F}_p) \arrow[r,"\operatorname{Cor}_G"]& H^1(G,\mathbb{F}_p) \arrow[r,"\chi \cup"]& H^2(G,\mathbb{F}_p)  \arrow[r,"\operatorname{Res}_{\ker(\chi)}"]&[1.2em]  H^2(\ker(\chi),\mathbb{F}_p) 
	\end{tikzcd}

We say that a profinite group $G$ which acts trivial on $\mathbb{F}_p$ admits the cup product exact sequence property (CPESP) if for every $\chi\in H^1(G, \mathbb{F}_p)$ the above sequence is exact. Thus, every absolute Galois group admits the CPESP. 
\begin{rem}
	Notice that the cup product sequence is always a complex.
\end{rem}
\begin{thm}\label{Mian_theorem_3}
Let $G$ be a profinite group acting trivially on $\mathbb{F}_p$. Assume that every closed subgroup of $G$ admits the CPESP. Then the same holds for $\hat{G}$.
\end{thm}
\begin{lem}
	It is enough to prove the property for every open subgroup of $\hat{G}$. 
\end{lem}
\begin{proof}
	Let $H$ be a closed subgroup of $G$ and $\chi\in H^1(H,\mathbb{F}_p)$. Recall that since $G$ acts trivially on $\mathbb{F}_p$, $H^1(H,\mathbb{F}_p)\cong \operatorname{Hom}(H,\mathbb{F}_p)$. Thus, by Lemma \ref{from_closed_to_open} $\chi$ can be lifted to some open subgroup $U\leq_oG$. Let $\rho\in H^1(H,\mathbb{F}_p)$ be an element such that $\chi\cup\rho=0$. By Lemma \ref{from_closed_to_open} again there is an open subgroup $U'$ such that $\rho$ can be lifted to $U'$. $\chi\cup\rho=0$ means that there is some $f\in C^1(H,\mathbb{F}_p)$ such that $\partial(f)=\chi\cup\rho$. Since $f$ is a continuous function from a profinite space to a finite space, it projects through some finite quotient of $H$ (see \cite{ribes2000profinite} Proposition 1.1.16). I.e, there exists a finite group $A$, a continuous homomorphism $f':H\to A$ and a continuous function $f'':A\to \mathbb{F}_p$ such that $f=f''\circ f'$. So again, there exists a lifting of $f'$, and thus of $f$, to an open subgroup $U''$ of $G$. Define $O=U\cap U'\cap U''$, we get that $O$ is an open subgroup of $G$ for which $\chi$ and $\rho$ can be lifted, denote these lifting by $\bar{\chi}$ and $\bar{\rho}$, and $\bar{\chi}\cup \bar{\rho}$ is trivial. By assumption, $\bar{\rho}=\operatorname{Cor}_O\rho'$ for some $\rho'\in H^1(\ker(\chi)_p),\mathbb{F}_p)$. Taking $\rho''=\operatorname{Res}_{H\cap \ker(\chi)}(\rho)$ will do.

For the second exactness, let $\alpha\in H^2(H,\mathbb{F}_p)$ such that $\operatorname{Res}_{\ker(\chi)}(\alpha)=0$. By the same arguments, there exists some open subgroup $O\leq_o G$ containing $H$ such that $\alpha$ and $\chi$ can be lifted to, and $\bar{\alpha}$, the lifting of $\alpha$, is trivial. By assumption, $\bar{\alpha}=\bar{\chi}\cup \rho$ for some $\rho\in H^1(O,\mathbb{F}_p)$, when $\bar{\chi}$ denotes the lifting of $\chi$ to $O$. Then $\operatorname{Res}_C(\rho)$ will prove the exactness. 
\end{proof}
\begin{proof}[Proof of Theorem \ref{Mian_theorem_3}]
	Let $\hat{U}$ be an open subgroup of $\hat{G}$, and $\chi\in H^1(\hat{U},\mathbb{F}_p)$ a continuous homomorphism. We will prove the exactness of the following series:
	
		\begin{tikzcd}
		H^1(\ker(\chi),\mathbb{F}_p) \arrow[r,"\operatorname{Cor}_{\hat{U}}"]& H^1(\hat{U},\mathbb{F}_p) \arrow[r,"\chi \cup"]& H^2(\hat{U},\mathbb{F}_p)  \arrow[r,"\operatorname{Res}_{\ker(\chi)}"]&[1.2em]  H^2(\ker(\chi),\mathbb{F}_p) 
	\end{tikzcd}
\begin{itemize}
	\item \textit{Exactness at $H^1(\hat{U},\mathbb{F}_p)$.} Let $\rho\in H^1(\hat{U},\mathbb{F}_p)$ be such that $\chi\cup\rho=0$. Let us look at the restrictions to $U$, these are not-necessarily-continuous homomorphisms $U\to \mathbb{F}_p$. By Nikolov\&Segal Theorem, for every f.g closed subgroup $H$ of $U$, $\operatorname{Res}_H(\chi)$ and $\operatorname{Res}_H(\rho)$ are continuous homomorphsims. Moreover, taking the restriction of $f$ to $H$, for $f\in C^1(\hat{U},\mathbb{F}_p)$ such that $\partial(f)=\chi\cup\rho$ implies that $\operatorname{Res}_H(\chi)\cup\operatorname{Res}_H(\rho)=0$. So by assumption $\operatorname{Res}_H(\rho)=\operatorname{Cor}_H(\rho'_H)$ for some $\rho'\in H^1(\ker(\operatorname{Res}_H(\chi)),\mathbb{F}_p)$. Since $H$ is f.g then $H\cap \ker(\chi)$ is open in $H$. Thus, by \cite[Corollary 2.5.5]{ribes2000profinite} $H\cap \ker(\chi)$ is f.g too.  Moreover, if $H\leq K$ are finitely generated closed subgroup of $U$ then for all $\rho'_K$ such that $ \operatorname{Res}_K(\rho)=\operatorname{Cor}_K(\rho'_K)$ one has $ \operatorname{Res}_H(\rho)=\operatorname{Cor}_H(\operatorname{Res}_H(\rho'_K))$, and this property also preserved under union, so by Proposition \ref{key tool} there is an element $\rho'\in H^1_{abs}(\ker(\chi|_U),\mathbb{F}_p)$ which is an inverse limits of all these homomorphosms. This follows from the fact that $U$ equals to the union of all its f.g closed subgroups $H$ and thus $\ker(\chi|_U)=\bigcup \ker(\operatorname{Res}_H(\chi))$. Moreover, it is easy to see that $\rho=\operatorname{Cor}_U(\rho')$. Since $H^1_{abs}(U,\mathbb{F}_p)=\operatorname{Hom}_{abs}(U,\mathbb{F}_p)$, $\rho'$ can be lifted to a continuous homomorphism $\hat{\rho'}:\hat{\ker(\chi)}\to \mathbb{F}_p$. The only thing left to show is that $\rho=\operatorname{Cor}_{\hat{U}}(\hat{\rho'})$. For that choose a transversal $\{u_1,...,u_n\}$ of $\ker(\chi|_U)$ in $U$. By [Proposition 3.2.2]\cite{ribes2000profinite} this is also a transversal of $\ker(\chi)$ in $U$. Recall that for each $x\in \hat{U}$ $\operatorname{Cor}_{\hat{U}}(\hat{\rho'})(x)=\sum_{u_i}\hat{\rho'}(u_i^{-1}x\widetilde{x^{-1}u_i})$ where $ \widetilde{x^{-1}u_i}$ denotes the representative of the coset of $x^{-1}u_i$ from the chosen set of transversal. An analogue of this interpretation holds for $\operatorname{Cor}_U(\rho')$. So $\rho$ and $ \operatorname{Cor}_{\hat{U}}(\hat{\rho'})$ identify on the dense subset $U$, and we are done. 
	
	\textit{Exactness at $H^2(\hat{U},\mathbb{F}_p)$}. Let $\alpha\in H^2(\hat{U},\mathbb{F}_p)$  be such that $\operatorname{Res}_{\ker(\chi)}(\alpha)=0$. Its restriction to $U$ is an element in the second abstract cohomology of $U$ whose restriction to $\ker(\chi|_U)$ is trivial. Since $\alpha$ is continuous, there is some finite continuous image $A$ of $\hat{U}$ such that $\alpha$ splits through the natural projection $\hat{U}\times\hat{U}\to A\times A$ (\cite[Proposition 1.1.6]{ribes2000profinite}). I.e, $\alpha=\operatorname{Inf}_{\hat{U}}(\alpha_A)$ where $\alpha_A\in H^2(A,\mathbb{F}_p)$. Thus $\operatorname{Res}_U(\alpha)$ is also inflated from $A$, and hence so is $\operatorname{Res}_H(\alpha)$ for every f.g closed subgroup $H$ of $U$. By Nikolov\&Segal Theorem we get that the reduced homomorphism $H\to A$ is continuous, so $\operatorname{Res}_H$ is a continuous cocycle. The same holds for $\operatorname{Res}_H(\chi)$. So by assumption there is some $\rho_H\in H^1(H,\mathbb{F}_p)$ such that $\operatorname{Res}_H(\chi)\cup \rho_H=\operatorname{Res}_H(\alpha)$. This property preserved under reduction and union, so by Proposition \ref{key tool} there is a cocycle $\rho\in H^1(U,\mathbb{F}_p)$ such that $\operatorname{Res}_U(\chi)\cup \rho=\operatorname{Res}_U\alpha$. Since $\rho$ is in fact a homomorphism it can be lifted to a continuous homomorphism $\hat{\rho}\in \operatorname{Hom}(U,\mathbb{F}_p)=H^1(U,\mathbb{F}_p)$. So the only thing left to show is that $\chi\cup\hat{\rho}=\alpha$. But this follows from their identification on the dense subset $U\times U$.  
\end{itemize}
\end{proof}
\subsection{Completions of higher order}
In \cite{BarOn2018tower} the author presented the infinite tower of profinite completions which is defined as follows: Let $G$ be a nonstrongly complete profinite group and $\alpha$ an ordinal. Denote $G_0=G$.

If $\alpha=\beta+1$ is a successor ordinal then $$G_{\alpha}=\widehat{G_{\beta}}$$ and if $\alpha$ is limit then $$G_{\alpha}=\widehat{H_{\alpha}}$$ when $H_{\alpha}=\widehat{{\dirlim}_{\beta<\alpha}G_{\beta}}$. The chain is equipped with natural competible homomorphisms $\varphi_{\beta\alpha}:G_{\beta}\to G_{\alpha}$. It has been proven that all the groups in this chain are nonstrongly complete and the maps are injective, and thus this chain grows indefinetely. In \cite{BarOn+2021} the local weights of all the groups in the chain have been computed and discovered to stricktly increase. In this paper we treat to $G_{\alpha}$ as the $\alpha$'th completion of $G$ and we want to prove that if $G$ possess some property of absolute Galois groups that was stated above, then so does $G_{\alpha}$ for every ordinal $\alpha$.  First we need the following lemma:  
	\begin{lem}\label{locally profinite}
		Let $\alpha$ be a limit ordinal and $U$ be a finite index subgroup of $H_{\alpha}=\widehat{{\dirlim}_{\beta<\alpha}G_{\beta}}$. Then $U$ is locally profinite and more precisely every f.g abstract subgroup of $U$ is contained in a closed subgroup of some $G_{\beta}$ for $\beta<\alpha$ which in turn is contained in $U$.
	\end{lem}
\begin{proof}
Let $x_1,...,x_n\in U$. There is some $\beta<\alpha$ such that $x_1,...,x_n\in G_{\beta}$. Thus $x_1,...,x_n\in U\cap G_{\beta}$ which is a finite index subgroup of $G_{\beta}$. So, by Lemma \ref{locally_closed} the closed subgroup generated by $x_1,..,x_n$ in $G_{\beta}$ is contained in $U\cap G_{\beta}$. Since $\varphi_{\beta\alpha}$ is injective it is also included in $U$. 
\end{proof}
\begin{prop}
	Let $\mathcal{C}$ be a property of profinite groups which satisfies the following conditions:
	\begin{itemize}
		\item Every closed subgroup of $G$ satisfies property $\mathcal{C}$ if and only if every open subgroup of $G$ does so.
		\item If $U$ is an abstract group for which every finitely generated abstract group is contained in a finitely generated profinite group satisfies $\mathcal{C}$ then $\hat{U}$ satisfies $\mathcal{C}$.
	\end{itemize}
then if $G$ is a nonstrongly complete profinite group for which every closed subgroup satisfies $\mathcal{C}$, then the same hold for $G_{\alpha}$, the $\alpha$'th completion of $G$, for every ordinal $\alpha$.
\end{prop}
\begin{proof}\label{main prop}
	We prove it by transfinite induction. Let $\alpha$ be an ordinal.
	
	$\alpha=0$: That's the assumption.
	
	$\alpha=\beta+1$ is successor. By induction assumption every closed subgroup of $G_{\beta}$ satisfies $\mathcal{C}$. By Lemma \ref{locally_closed} every finite index subgroup $U$ of $G$ is locally closed and thus satisfies the conditions of Condition 2, so $\hat{U}$ satisfies $\mathcal{C}$. Thus by Condition 1 every closed subgroup of $G_{\alpha}=\hat{G_{\beta}}$ satisfies $\mathcal{C}$. 
	
	$\alpha$ is limit: Let $U$ be a finite index subgroup of $H_{\alpha}$. By Lemma \ref{locally profinite} every finitely generated subgroup of $U$ is contained in a closed subgroup of some $G_{\beta}$ for $\beta<\alpha$. By induction assumption these closed subgroup satisfy property $\mathcal{C}$. Thus $U$ satisfies the conditions of Condition 2 so $\hat{U}$ satisfies $\mathcal{C}$. Thus by Condition 1 every closed subgroup of $G_{\alpha}=\hat{H_{\alpha}}$ satisfies $\mathcal{C}$. 
\end{proof}
\begin{thm}
	Let $G$ be a nonstrongly complete profinite group which satisfies one of the following options:
	\begin{itemize}
		\item $G$ is projective.
		\item $G$ acts trivially on $\mathbb{F}_p$ and every closed subgroup of $G$ satisfies the $n$- vanishing Massey product property.
		\item $G$ has an orientation that induces the trivial module structure on $\mathbb{F}_p$ and every closed subgroup of $G$ is cyclotomic with regard to this orientation.
		\item $G$ act trivially on $\mathbb{F}_p$ and every closed subgroup of $G$ satisfies the cup-product exact sequence property.
	\end{itemize}
Then the same holds for $G_{\alpha}$, correspondingly, for every ordinal $\alpha$.
\end{thm}
\begin{proof}
	This is an immediate consequence of Proposition \ref{main prop} and the proofs of Theorems \ref{main theorm 1}, \ref{Main theorem 2}, \ref{Mian_theorem_3} and the main theorem of \cite{bar-on_2021}. The only thing left to show is that the orientation of $G$ can be lifted to an orientation for every $G_{\alpha}$ in the chain, but this is also immediate since by the definition of profinite completion every homomorphism to a profinite group can be lifted to a compatible homomorphism from the completion, and compatible homomorphism induces a homomorhism from the direct limit, by the universal property of direct limit.    
\end{proof}

	\bibliographystyle{plain}
	\bibliography{../../Results_for_PHD/references}
\end{document}